\documentclass[bsl,bibalpha]{asl}

\usepackage{
amsfonts,
latexsym,
amssymb,
amscd,
enumerate,
}

\newcommand{\labbel}{\label}

\newcommand{\myarrow}[3]{{#1}\Rightarrow^{{#2}}{#3}}

\newcommand{\myarrownot}[3]{{#1}\not\Rightarrow^{{#2}}{#3}}

\newtheorem{theorem}{Theorem}[section]
\newtheorem{lemma}[theorem]{Lemma}
\newtheorem{thm}[theorem]{Theorem}
 
\newtheorem{proposition}[theorem]{Proposition} 
 
\newtheorem{corollary}[theorem]{Corollary} 
\newtheorem{fact}[theorem]{Fact} 
\newtheorem*{corollary*}{Corollary}

\theoremstyle{definition}
\newtheorem{definition}[theorem]{Definition}

\theoremstyle{remark}

\newtheorem{remarks}[theorem]{Remarks}

\newcommand{\brfrt}{\hspace{0 pt}}

\DeclareMathOperator{\cf}{cf}
\DeclareMathOperator{\CAP}{CAP}

\newcommand{\m}{\mathfrak}

\begin{document}

\title{Weak and local versions of measurability}

\author{Paolo Lipparini} 
\revauthor{Lipparini, Paolo} 
\address{Dipartimento di Matematica\\ Viale della Ricarica
 Scientifica \\II Universit\`a di Roma (Tor Vergata)\\I-00133 ROME ITALY}
\urladdr{http://www.mat.uniroma2.it/\textasciitilde lipparin}


\thanks{Work performed under the auspices of G.N.S.A.G.A}

\begin{abstract}
Local versions of measurability have been around for a long time.
Roughly, one splits the notion of $\mu $-completeness into pieces, and asks
for a uniform ultrafilter over $\mu $ satisfying just some piece of
$\mu $-completeness. 

Analogue local versions of weak compactness are harder to come by, since
weak compactness cannot be defined by using a single ultrafilter.
We deal with the problem by restricting just to a subset $P$ of 
all the partitions of $\mu $ into
$<\mu $ classes and asking for some ultrafilter $D$ over $\mu $ such that 
no partition in $P$ 
disproves the $\mu $-completeness of $D$.
By making $P$ vary in appropriate classes,
one gets both  measurability and 
weak compactness, as well as possible intermediate notions
of ``weak measurability''.

We systematize the above procedures and combine them
to obtain variants of measurability which are at the same time weaker and local.
Of particular interest is the fact that the notions thus obtained
admit equivalent formulations through  topological, model theoretical, combinatorial
and Boolean algebraic conditions. We also hint a connection
with Kat{\v{e}}tov order on filters.
\end{abstract} 
 
\maketitle

\section{Introduction} \labbel{intro} 

\subsection{Local forms of measurability} \labbel{locsub} 
Local versions of measurability have been considered by many authors, 
among them Chang \cite{cha67}, Prikry \cite{pri73}, Silver \cite{sil74}, 
just to state some.

To cast our introduction into a general framework, let us recall some definitions.
If $D$ is an ultrafilter over $\mu$, then $D$ is \emph{$ \lambda $-decomposable}
if there is a partition
of $\mu $ into $\lambda$ many classes in such a way that
no union of $<\lambda$  classes belongs to $D$.
Equivalently (\ref{rmkmeas}(1)), $D$ is $ \lambda $-decomposable
if  there is a function $f: \mu   \to \lambda $ such that $f(D)$ is \emph{uniform over $ \lambda  $},
that is, every member of $f(D)$ has cardinality $ \lambda $.
Here $f(D)$ is the ultrafilter over $ \lambda $ defined by
$ Y \in f(D) $ if and only if  $f ^{-1}(Y) \in D $.
The relation induced by this  ``quotient'' operation is usually
 called the \emph{Rudin-Keisler (pre-)order}, thus
an ultrafilter $D$ is $\lambda$-decomposable if and only if 
there is some ultrafilter uniform over $\lambda$ and $\leq D$ 
modulo the Rudin-Keisler order.  
Each of the above points of view---through quotients and through 
partitions---has its own advantages, as we shall discuss
in details.
We say that $D$ is \emph{$ \lambda $-indecomposable}
if it is not $ \lambda $-decomposable.

If $\mu $ is measurable, then trivially there is a uniform ultrafilter $D$ 
over $\mu $ which is $ \lambda $-indecomposable for every $\lambda< \mu$. 
Chang \cite{cha67} first noticed that the study of $\lambda$-decomposability
is not trivial, even for small cardinals. Actually, Chang dealt with
the notion of \emph{$\lambda$-descending incompleteness}, which is nevertheless
equivalent to $\lambda$-decomposability, for $\lambda$ regular. See
\cite{mru} for this and many other equivalences, and for an exhaustive list of references
to the subject, which includes works by  the mentioned authors and by 
 A. Adler,
 A. W. Apter, 
 S. Ben-David,
M. Benda, 
G. V. {\v{C}}udnovski{\u\i},
 D. V. {\v{C}}udnovski{\u\i}, 
H.-D. Donder, 
M. Foreman, 
 J. M. Henle, 
 M. Huberich, 
T. Jech,  
R. B. Jensen,  
M. Jorgensen,
A. Kanamori, 
J. Keisler, 
J. Ketonen, 
 B. J. Koppelberg,
K. Kunen,
M. Magidor, 
M. Sheard, 
A. D. Taylor and
W. H. Woodin.

In fact, an uncountable cardinal $ \mu $ is measurable
\emph{if and only if}  it carries a uniform ultrafilter which is 
$\lambda$-indecomposable for every $\lambda< \mu$.  
More generally, an ultrafilter 
is \emph{$\mu $-complete}
if and only if it is
 $\lambda$-indecomposable for every $\lambda< \mu$. 
 See Remark \ref{rmkmeas}(2) here, or  \cite{mru} for further details. 
Thus
ultrafilters enjoying  various degrees of indecomposability 
 furnish weaker ``local'' analogues  of measurability,
The existence 
of a $ \lambda $-indecomposable  uniform ultrafilter over $\mu $
admits many equivalent formulations in terms of
topology, model-theory (both first order and extended), infinite combinatorics,
Boolean algebras.
 See \cite{bumi,arch,nuotop,cpnw,cpnwb} for examples. 
In a sense, the situation is similar
to the one described in the classical paper \cite{KT}
by H. J. Keisler and A Tarski, where a big deal of conditions
equivalent to a cardinal being measurable have been worked out.
The existence 
of a $ \lambda $-indecomposable uniform ultrafilter over $\mu $,
as a local version of the measurability of $\mu $, and just as
measurability itself, can be expressed in equivalent forms 
in a varied sets of mathematical frameworks.
Indeed, 
the mentioned topological, model theoretical, etc., characterizations
of the existence of a  $ \lambda $-indecomposable uniform ultrafilter over $\mu $ can be easily generalized in order to 
provide, 
for every set $\Lambda$ of cardinals, 
characterizations of the existence
of a uniform ultrafilter over $\mu $ which for no $ \lambda \in \Lambda$ 
is $\lambda$-decomposable.
When $\Lambda$  is  the set of all cardinals
$< \mu $ 
one usually
recovers exactly the conditions from \cite{KT}.

\subsection{Weak compactness with and without inaccessibility} \labbel{wcwwi}

The mentioned paper by Keisler and Tarski 
contains also many topological, model-theoretical, etc.,
equivalent formulations of
all the large cardinals known at that time,
in particular, also of
 weak compactness. It is not apparent how to get
``local'' versions of such results in a sense parallel to the above 
local versions of measurability.
In this case, the situation is made even more difficult by the fact
that there are many definitions of weak compactness 
which turn out to be equivalent only under the assumption
that the cardinals at hand are inaccessible
(by \emph{inaccessible} we shall always mean, as nowadays usual,
\emph{strongly inaccessible}).
Though conventional wisdom has  suggested 
that inaccessibility should be included right in the definition 
of weak compactness, we believe that in such a way
the richness of many interesting phenomenons gets lost.

We refer, for example, to earlier 
topological and measure-theoretical
studies by
Mr{\'o}wka \cite{mro66,mro70},
followed by {\v{C}}udnovski{\u\i} \cite{cud72}, 
or to a homological theorem 
by A. Mekler  mentioned in Eklof \cite[Theorem 1.6]{ekl77}. 
All these results are partially trivialized by the assumption of
inaccessibility.
Also
the \emph{tree property}  
is  equivalent to weak compactness 
only under the assumption of inaccessibility,
and even  a successor cardinal may satisfy 
it, Mitchell \cite{Mit73}. 
Variants of the tree property
have received a notable attention
in recent years. See, e.~g., Fontanella \cite{Fon13},
 Viale and Wei{\ss} \cite{VW11}, 
 Wei{\ss} \cite{Wei12} and further references there.
However the properties we are considering here
differ from the tree property and from its variants in that
they do imply weak inaccessibility.
That strong inaccessibility is
effectively necessary for the
equivalence of most variants of
weak compactness  has been verified by Boos \cite{Boo76},
who constructed models in which many such variants
do not imply (strong) inaccessibility.
At the same time, in \cite[Theorem 2.4]{Boo76} Boos  also  finds 
further equivalent conditions for 
weak compactness without inaccessibility.

Subsequently in \cite{bumi} we showed that weak compactness 
without inaccessibility deeply affects  the study of extended 
logics. In the present note  logics 
will enter the scene mostly as examples.
Roughly, the reader might think of a \emph{logic} 
as an extension of first order logic which satisfies all
the properties common both 
 to infinitary logics and to logics with added quantifiers,
e.~g., the quantifier $Q_ \alpha $ saying ``there are $\aleph_ \alpha$ many''.
An extensive review of the subject can be found in \cite{BF};
a concise survey of what happened in the last years can be found in \cite{she12}.
Exactly in the same way as weakly compact cardinals can be defined
as  cardinals for which the corresponding infinitary logic satisfies compactness,
we showed in \cite{bumi} that, for \emph{every} logic $\mathcal N$,
the first cardinal $\kappa$ such that $\mathcal N$ is $( \kappa , \kappa )$-compact
 is weakly compact in the weaker sense that 
$\mathcal L _{ \kappa , \omega } $ is
 $( \kappa , \kappa )$-compact. In case 
$\kappa > \omega $ is inaccessible, this is one
among the many possible definitions of 
weak compactness.
Thus a cardinal $\kappa$ is weakly compact (in the above weaker sense, and including $ \omega$) if and only if it is the first cardinal
such that some logic is  $( \kappa , \kappa )$-compact.
Recall that  $\mathcal L _{ \kappa , \nu } $ is like first order logic, except that
conjunctions and disjunctions of $<\kappa$ sentences  are allowed,
as well as simultaneous universal or existential quantification
over sets of $<\nu$ variables.
A logic is \emph{$( \kappa , \kappa )$-compact}
if every $\kappa$-satisfiable set $ \{ \sigma_ \alpha \mid  \alpha \in \kappa \}  $
  of $\kappa$  sentences is
satisfiable.
A set is \emph{$\kappa$-satisfiable} 
if every subset of cardinality $<\kappa$ has a model.  
The results from \cite{bumi} not only support the conviction
that weak compactness without inaccessibility is interesting and deserves further study,
but also show that, say, the  Betelgeusians would 
have arrived at the very same notions of weak and strong compactness
as ours, even had they started by considering
 entirely different logics.
E~g., the first  weakly compact 
(in the above weaker sense) cardinal $\kappa$ is
the first cardinal such that $\mathcal L(Q_0) $
is $( \kappa , \kappa )$-compact.

\subsection{Local versions of weak compactness} \labbel{locwc} 

As we mentioned, 
it is not apparent how to introduce local versions of weak compactness 
analogue to the local versions of measurability
described in Subsection \ref{locsub}, since
weak compactness cannot be defined by using a single ultrafilter.
We originally devised a model-theoretical approach to the problem in \cite{bumi},
again motivated by extended logics.
However, subsequently we found an ostensibly  simpler method 
 which uses the following observation,
bearing some similarity with Mr{\'o}wka ideas from \cite{mro66,mro70}.
Clearly, an \emph{ultrafilter} $D$ over $\mu $ 
is $\mu $-complete if and only if 
whenever we partition  $\mu $ into $<\mu $ classes,
one of these classes is in $D$. 
Hence checking $\mu $-completeness of some ultrafilter over $\mu $ 
amounts to check it for all the $2^\mu $ partitions of $\mu $ into
$<\mu $ classes. We can pick a subset $P$ of all such partitions
and only ask that 
there is some ultrafilter $D$ such that 
none of the above partitions is a witness for the $\mu $-incompleteness
of $D$
(the $\mu $-incompleteness of $D$ might be or might be not disproved  by some  partition outside  P). Taking $P$ to be the set of all the above partitions 
 gives back the notion of measurability,
while making $P$ vary among sets of cardinality
$\mu $ gives an equivalent formulation of weak compactness.
If GCH fails at $\mu $, 
we can consider all $P$'s of cardinality $\nu$,
for $\nu$ strictly between
 $\mu $ and $  2^ \mu $.
This provides weak versions of measurability and
Schanker \cite{sch11}, using
an equivalent formulation,  recently proved that there are models in 
which such intermediate notions are actually
distinct both from measurability and from weak compactness.
 See also 
\cite{CGHS}.

Pursuing further the idea that $\mu $-completeness is ``composed'' of
pieces of $\lambda$-indecomposability for $\lambda < \mu$,
as described in Subsection \ref{locsub}, 
and using the observations in the above paragraph,
we can introduce a form of ``local weak compactness at $\lambda$''
by asking that $\kappa$ many partitions of $\mu $ are not enough
to witness that every uniform ultrafilter over $\mu $ is $\lambda$-decomposable.
Let us denoted by $ \myarrownot {\mu}{\kappa}{\lambda} $
the above statement, intended to mean that \emph{it is not the case}  that
``uniformity on $\mu $ implies (by means of $\kappa$ many partitions) 
$\lambda$-decomposability''.
Correspondingly, the above statement between quotes
will be denoted by $ \myarrow {\mu}{\kappa}{\lambda} $. 
Notice that it might happen  that there is indeed 
a $\lambda$-indecomposable
 uniform ultrafilter over $\mu $ 
(a local version of measurability at $\lambda$),
in which case $ \myarrownot {\mu}{\kappa}{\lambda} $ for every $\kappa$.
On the other hand, it is possible that 
every uniform ultrafilter over $\mu $ is $\lambda$-decomposable
(``local measurability''  at $\lambda$ fails)
but perhaps all of the $2^\mu $ partitions of $\mu $ 
into $\lambda$ pieces are necessary to witness this (or, at least, $2^\mu $ 
many partitions are needed). 
In this case 
$ \myarrow {\mu}{ 2^ \mu }{\lambda} $,
but
$ \myarrownot {\mu}{\kappa}{\lambda} $ for every 
$\kappa < 2^\mu$. We also may have intermediate cases.
 We interpret 
$ \myarrownot {\mu}{\mu}{\lambda} $ 
 as a local version
of weak compactness at $\lambda$,
while if $\kappa$ grows larger in
$ \myarrow {\mu}{\kappa}{\lambda} $
we go closer and closer to (local versions) of measurability.
We sometimes find it convenient to work with functions
rather than with partitions (cf. the two equivalent definitions of
$\lambda$-decomposability given at the beginning).
The definition of 
$ \myarrow {\mu}{\kappa}{\lambda} $ 
by means of functions
is given in Definition \ref{def}; 
in Lemma \ref{partitions2}(2) it is proved 
 equivalent to the 
definition in terms of partitions sketched above.

In the above discussion 
we can consider a set $ \Lambda$ 
of cardinals in place of $\lambda$,
and 
introduce a similar principle 
$ \myarrow {\mu}{\kappa}{\Lambda} $ 
which says  that $\lambda$-\brfrt decomposability is
witnessed for at least one $ \lambda \in \Lambda$
(Definition \ref{defL}).
When we take $\Lambda = Card _{< \mu} $,
the set of all infinite cardinals $<\mu $,
then $ \myarrownot {\mu}{2 ^ \mu}{\Lambda} $
corresponds exactly to the measurability of $\mu $, 
while
$ \myarrownot {\mu}{\mu}{\Lambda} $ 
 corresponds  to weak compactness of $\mu $ 
(provided we  assume  either that $\mu $  is inaccessible, or 
that we are dealing with weaker versions of weak compactness,
as described in the previous subsection).
The situation is represented  in the following  table,
where we write 
$C _{< \lambda } $ in place of  $Card _{< \lambda } $ 
to save space.
\begin{gather*}      \labbel{table1}
\textsc{  Table 1  }
\\
\text{($\mu \stackrel{ \kappa  }{\not\Rightarrow} \Lambda $, measurability and weak compactness)}
\\ 
\textsc{\small Local versions\ \ \ \ 
\parbox[c][2em][c]{0.3\textwidth}{\centering 
\footnotesize (going more and\\ more global)}\ \ \ \ \ 
\small Global versions}
\\
\begin{CD} 
\parbox[c][2em][c]{0.25\textwidth}{\centering \footnotesize (a $\lambda$-indecomp.     \\ ultrafilter over $\mu $)}
@.
\mu \stackrel{ 2^\mu }{\not\Rightarrow} \lambda
@>
{\text{$\Lambda$ getting larger}}
>
\parbox[c][0.5em][c]{0.25\textwidth}{\centering \scriptsize (going closer to meas.)}
>
 \mu \stackrel{ 2^\mu }{\not\Rightarrow} C_{< \mu}  
@.
 \text{\centering \footnotesize ($\mu $  measurable)}
\\
@. 
@A
\parbox[c][2em][c]{0.23\textwidth}{\centering \scriptsize 
$\kappa$ increasing\\
(going closer to local \\ 
measurability at $\lambda$)}
A
A   
@A
\parbox[c][2em][c]{0.18\textwidth}{\centering \scriptsize (going closer to measurability)}
A
A
 \text{\centering \footnotesize ($\mu $ weakly meas.)}
\\
\parbox[c][4em][c]{0.25\textwidth}{\centering \footnotesize ($\mu $ locally weakly compact at $\lambda$)} 
@.
\mu \stackrel{ \mu }{\not\Rightarrow} \lambda
@>
>
\parbox[c][1em][c]{0.25\textwidth}{\centering \scriptsize (going closer to \\  weak compactness)}
>
 \mu \stackrel{ \mu }{\not\Rightarrow}C _{< \mu} 
@. 
\parbox[c][4em][c]{0.23\textwidth}{\centering \footnotesize ($\mu $  weakly compact
\\ if inaccessible)}
\end{CD}   
\end{gather*}
\subsection{Topological and model-theoretical  equivalences} \labbel{locwc2} 

There are pleasant topological characterizations 
of the relations 
$ \myarrow {\mu}{\kappa}{\lambda} $ 
and 
$ \myarrow {\mu}{\kappa}{\Lambda} $.
A topological space  $X$ is \emph{$\mu $-compact}
if every subset of cardinality $\mu $ has a complete accumulation point.
Clearly, if $|X| < \mu$ then $X$ is vacuously $\mu $-compact.
On the other hand, if $X$ is a product of cardinals and
$|X| \geq \mu $, then $\mu $-compactness of $X$
depends on the relations introduced above. 

For example, if $\lambda$ is regular and
$|\lambda^ \kappa | \geq \mu $ then 
$ \myarrownot {\mu}{\kappa}{\lambda} $  
if and only if 
$\lambda^ \kappa $
is  
$\mu$-compact.
Here $\lambda$ is considered as a topological space 
endowed with the \emph{order topology},
 a base of which consists of all open intervals,
including intervals of the form $[0, \alpha )$.  
Powers and products are always endowed with the
Tychonoff topology, the coarser topology which
makes the projections continuous.

More generally, under similar cardinality and regularity assumptions,
$ \myarrownot {\mu}{\kappa}{\Lambda} $  
if and only if 
$ \prod _{ \lambda \in \Lambda } \lambda^ \kappa $
is  
$\mu$-compact.
We can draw again a diagram.
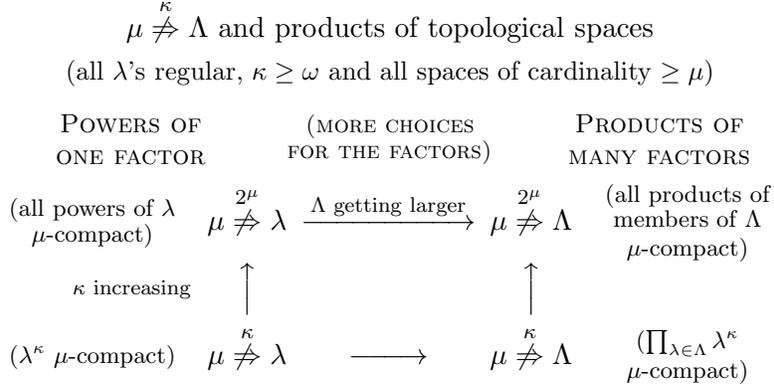
\begin{gather*}      \labbel{table2}
\textsc{  Table 2  }
\\
\text{$\mu \stackrel{ \kappa  }{\not\Rightarrow} \Lambda $ and products of topological spaces}
\\ 
\text{\small (all $\lambda$'s  regular, $\kappa \geq \omega $ and all spaces of cardinality $\geq \mu $)}
\\
\textsc{ 
\parbox[c][2.3em][c]{0.27\textwidth}
{\centering  \small Powers of  \\ one factor}
\parbox[c][2.3em][c]{0.27\textwidth}{\centering 
\footnotesize (more choices \\ for the factors)}
\parbox[c][2.3em][c]{0.27\textwidth}
{\centering  \small \ \ Products of  \\ \ \ many factors}
}
\\
\begin{CD} 
\parbox[c][2em][c]{0.25\textwidth}{\centering \footnotesize (all powers of $\lambda$ \\ $\mu $-compact)}
@.
\mu \stackrel{ 2^\mu }{\not\Rightarrow} \lambda
@>
{\text{$\Lambda$ getting larger}}
>
>
 \mu \stackrel{ 2^\mu }{\not\Rightarrow} \Lambda   
@.
\parbox[c][2em][c]{0.25\textwidth}{\centering \footnotesize (all products of members
 of $\Lambda$ $\mu $-compact)}
\\
@. 
@A
\parbox[c][2em][c]{0.23\textwidth}{\centering \scriptsize 
$\kappa$ increasing}
A
A   
@A
A
A
\\
\parbox[c][2em][c]{0.25\textwidth}{\centering \footnotesize ($\lambda^ \kappa $ $\mu $-compact)} 
@.
\mu \stackrel{ \kappa }{\not\Rightarrow} \lambda
@>
>
>
 \mu \stackrel{ \kappa  }{\not\Rightarrow} \Lambda  
@. 
\parbox[c][2em][c]{0.25\textwidth}{\centering \footnotesize ($ \prod _{ \lambda \in \Lambda} \lambda ^ \kappa $ $\mu $-compact)}
\end{CD}   
\end{gather*}

Model theoretical equivalents are dealt with in Section \ref{mt}.

\subsection{Note} \labbel{note} 

This is a preliminary version. More results and observations are planned
to be added in the future. Results related to the present work have been presented, 
proved or announced in the previously quoted papers and, e.~g., in
\cite{pamsuf,tapp,tproc2},
sometimes in equivalent formulations, or with slightly different notations.

\section{Basic definitions} \labbel{bas} 

Our notation is standard and,
if not mentioned otherwise, follows \cite{J}.
Throughout, $\alpha$, $\beta$, $\gamma$ are ordinals, $\lambda$, $\mu $, $\nu$, $\xi$ are infinite cardinals,
 $\kappa$ and $\theta$ are nonzero cardinals, $J$ is a nonempty set and
$\Lambda$, $\Upsilon $ are nonempty sets of infinite cardinals.
When there is no risk of ambiguity we shall write, say,
$\lambda \geq \mu ,\nu$ as a shorthand for 
$\lambda \geq \sup \{ \mu ,\nu \} $. 
$Card$ and
$Reg$ denote, respectively, the class of all infinite cardinals and of  all infinite regular cardinals.
$Card[ \lambda ,\mu] $ is $\{ \nu \mid \lambda \leq \nu \leq \mu \}$, and
$Card(\lambda ,\mu) $, $Card[ \lambda ,\mu) $ have a similar meaning.
We write $Card[ \lambda ,\mu] $ in place of simply writing $[ \lambda ,\mu] $
both 
in order to avoid possible confusion with other notions (e.~g., compactness of topological spaces or of logics) and to make clear that members
of $Card[ \lambda ,\mu] $ are cardinals, rather than, say, ordinals.
$Reg[ \lambda ,\mu] = Card[ \lambda ,\mu] \cap Reg$, and similarly for
$Reg[ \lambda ,\mu)$.
We sometimes write $Card _{< \mu}  $ in place of $Card[ \omega ,\mu) $ and
$Reg _{< \mu}  $ in place of $Reg[ \omega ,\mu)$. 

\begin{definition} \labbel{def}    
We denote by
$\mu \Rightarrow ( \lambda _ j ) _{ j \in J } $ 
the following statement.
 \begin{enumerate}   
\item[(*)]
There is a sequence  
$(f_ j ) _{ j \in J }$
of functions 
$f_ j  : \mu  \to \lambda _ j  $
for $j \in J$, such that 
for every  uniform ultrafilter $D$ over $ \mu $ 
there is $j \in J$ such that 
$f_ j (D)$ is uniform over $ \lambda _ j$.    
 \end{enumerate}

In case (*) holds for some given sequence of functions
$(f_ j ) _{ j \in J }$ we shall say that the $f_ j $'s
\emph{witness} $\mu \Rightarrow ( \lambda _ j ) _{ j \in J } $.

We write 
$ \myarrow {\mu}{\kappa}{\lambda} $ 
when $|J |= \kappa $ and all the $ \lambda _ j $'s in (*) are equal to $ \lambda $.

The negations of
the above principles shall be denoted 
by 
$\mu {\not\Rightarrow} ( \lambda _ j ) _{ j \in J } $ and  
$ \myarrownot {\mu}{\kappa}{\lambda} $ 
respectively.
\end{definition}

Trivial facts about 
$\mu \Rightarrow ( \lambda _ j ) _{ j \in J } $ 
are that the notion is invariant under a permutation
of the indices, and that 
it is preserved by taking supersequences.
By this we mean that if 
$J' \subseteq J$ and the  
subsequence 
$(f_ j ) _{ j \in J' }$
witnesses $\mu \Rightarrow ( \lambda _ j ) _{ j \in J' } $
then the sequence $(f_ j ) _{ j \in J}$
witnesses $\mu \Rightarrow ( \lambda _ j ) _{ j \in J } $.

Moreover, notice that if  $|f_j(\mu)|< \lambda _j$ 
then $f_j(D) $ is not uniform over 
$ \lambda _j$;
in particular, 
if $ \mu < \lambda_j $ then 
$f_j(D) $ is not uniform over 
$ \lambda _j$.
From this  we get  the following facts.

\begin{fact} \labbel{fact0}
(1) The sequence
$(f_ j ) _{ j \in J }$
witnesses $\mu \Rightarrow ( \lambda _ j ) _{ j \in J } $
if and only if the subsequence 
$(f_ j ) _{ j \in J' }$
witnesses $\mu \Rightarrow ( \lambda _ j ) _{ j \in J '} $,
where
$J'= \{ j \in J \mid |f_j(\mu)|= \lambda _j \} $. 

(2) In particular, $\mu \Rightarrow ( \lambda _ j ) _{ j \in J } $ 
if and only if 
$\mu \Rightarrow ( \lambda _ j ) _{ j \in J'' } $, where
 $J'' = \{ j \in J \mid  \mu \geq \lambda _j  \} $. 
 \end{fact}

In view of Fact \ref{fact0}(2)
  it is no loss of generality if
in Definition \ref{def} 
we assume that
$ \mu \geq\lambda_j $,  for every $j \in J$.
On the other hand, if some $\lambda_j$ equals $\mu $, 
then 
$\mu \Rightarrow ( \lambda _ j ) _{ j \in J } $
is trivially true, as witnessed by 
taking  $f_j$  to be the identity function, or just any injective function.
In conclusion, 
the principle 
$\mu \Rightarrow ( \lambda _ j ) _{ j \in J } $
is interesting only when 
$ \mu > \lambda_j $,  for every $j \in J$.

We are soon going to show that
a cardinal is measurable if and only if 
$\mu {\not\Rightarrow} ( \lambda _ j ) _{ j \in J } $,
where each cardinal $ \lambda  < \mu $
appears $2^\mu $  times in the sequence 
$( \lambda _ j ) _{ j \in J }$. 
The above observation is better proved
after (and justifies) the introduction of  some further 
more compact notation.

\begin{definition} \labbel{defL}    
If $\Lambda$ is a set of infinite cardinals,
we write
$ \myarrow {\mu}{\kappa}{\Lambda} $ 
if 
$\mu \Rightarrow ( \lambda _ j ) _{ j \in J } $ 
 holds in case  each $\lambda \in \Lambda $ 
appears exactly $\kappa$ times in the sequence 
$( \lambda _ j ) _{ j \in J } $.

More formally,  
if $ \Lambda =\{ \lambda _h \mid h \in H\}$ then
$ \myarrow {\mu}{\kappa}{\Lambda} $ 
means 
$\mu \Rightarrow ( \lambda _ j ) _{ j \in J } $,
where $J= H \times \kappa $ and   
$\lambda _j = \lambda _ h $,
whenever $j=(h, \gamma) $, $h \in H$, $\gamma \in \kappa $.  

Notice that
$ \myarrow {\mu}{\kappa}{ \{  \lambda\}} $ 
is the same as 
$ \myarrow {\mu}{\kappa}{\lambda} $.

By convention, 
$ \myarrow {\mu}{\kappa}{\emptyset} $ 
is always considered to be false (this will occur  
infrequently and only as the basis of some induction). 
\end{definition}

Trivially, as above, if
$ \myarrow {\mu}{\kappa}{\Lambda} $, 
$ \theta  \geq \kappa $ and $\Upsilon \supseteq \Lambda $ then
$ \myarrow {\mu}{\theta}{\Upsilon} $. 

\begin{fact} \labbel{fact1}
The following conditions are equivalent.
 \begin{enumerate}[(1)]   
\item 
$ \myarrow {\mu}{2 ^ \mu }{\Lambda} $.
\item
$ \myarrow {\mu}{\kappa}{\Lambda} $, for some 
(equivalently, all) $\kappa \geq 2^ \mu$.
\item
Every uniform ultrafilter over $\mu $ is
$\lambda$-decomposable, for some $\lambda \in \Lambda $.    
\item
Every  $\mu $-decomposable ultrafilter (over any set) is
$\lambda$-decomposable, for some $\lambda \in \Lambda $.     
 \end{enumerate} 
 \end{fact}   

\begin{proof}
By Fact \ref{fact0}(2),  without loss of generality $\mu \geq \lambda$,
for every $\lambda \in \Lambda $.
Then the equivalence of (1)-(3)  is  trivial from the definitions, 
since for every $ \lambda \in \Lambda$ 
there are 
exactly $ \lambda ^ \mu =2^\mu $ 
functions from $\mu $ to $\lambda$.

(4) $\Rightarrow $  (3) is trivial.

If (3) holds, and $D$ is 
a $\mu $-decomposable ultrafilter over, say, $I$, 
then there is some function
$g: I \to \mu $ such that $g(D)$ is uniform over
$\mu $. By (3) $g(D)$ is 
$\lambda$-decomposable, for some $\lambda \in \Lambda $. 
This  is witnessed by some
$f: \mu  \to \lambda $,
thus $g \circ f$ witnesses the 
 $\lambda$-decomposability of $D$.
\end{proof}

By Fact \ref{fact1}, the only interesting cases in 
$ \myarrow {\mu}{\kappa}{\Lambda} $ 
are when $\kappa \leq 2^ \mu$.
As $\kappa$ grows larger in
$ \myarrow {\mu}{\kappa}{\Lambda} $, 
we get a weaker notion, but at the point
$\kappa = 2^ \mu$ we already get the minimum strength.

The equivalence of (1) and (2) in Fact \ref{fact1}
justifies the notation
$ \myarrow {\mu}{\infty}{\Lambda} $ 
used in in \cite[Section 6]{tproc2} in place of
$ \myarrow {\mu}{2^\mu}{\Lambda} $. 
In \cite[Theorem 6.3 and Corollary 6.6]{tproc2} we have listed many results about 
$ \myarrow {\mu}{2^\mu}{\Lambda} $. 
All these results can be appropriately
generalized to the relation
$ \myarrow {\mu}{\kappa}{\Lambda} $, but in
a few cases we have not yet written down full details.

Notice also that, as discussed in \cite[p. 343]{mru}
in a parallel situation,
while perhaps Condition (3) in \ref{fact1}
might appear simpler and more intuitive than Condition (4),  
the latter is very useful.
For example, using (4)
one immediately gets that
$ \myarrow {\mu}{2^\mu}{\lambda} $ 
and 
$ \myarrow { \lambda }{ 2^ \lambda }{ \upsilon } $ 
imply
$ \myarrow {\mu}{2 ^ \mu}{ \upsilon } $.
This is not immediately obvious using (3).
The above observation shall be expanded in 
Proposition \ref{trans}. 

\begin{remarks} \labbel{rmkmeas}    
(1) It is elementary to see that the two definitions 
of $\lambda$-decomposability 
given in the introduction are equivalent.
Indeed, if $\lambda$-decomposability is witnessed by some partition, then 
any enumeration of its classes produces a function
with the desired properties.
Conversely, any function witnessing $\lambda$-decomposability
naturally gives rise to a partition, which satisfies the corresponding
conditions. 

Similarly, we shall show in 
Lemma \ref{partitions1}. 
that
$\mu \Rightarrow ( \lambda _ j ) _{ j \in J } $
can be reformulated in terms of partitions.

(2) As mentioned  
in the introduction,
a cardinal $\mu $ is measurable if and only if 
there is a uniform ultrafilter $D$ over $\mu $ which is $\lambda$-indecomposable,
for every $\lambda < \mu$. Indeed,  a $\mu $-complete
ultrafilter is trivially $\lambda$-\brfrt indecomposable for every $\lambda< \mu$.   
For the other direction, and in contrapositive  form,
if $D$ over $I$ is not $\mu $-complete, then,
using the maximality of $D$,
there is a partition 
of $I$ into $< \mu $ sets such that
no member of the partition is in $D$.
If we take such a partition of minimal cardinality $\lambda$,
then no union of $<\lambda$ classes of the partition is in $D$,
thus  $D$ is $\lambda$-decomposable.

(3)
By (2) and as a particular case of Fact \ref{fact1},  
 a cardinal $\mu $ is measurable 
if and only if 
$ \myarrownot {\mu}{2^ \mu }{Card _{< \mu }} $ 
where 
$ Card _{< \mu } $
denotes the set
of all infinite cardinals $<\mu $. 
\end{remarks}

The  just introduced arrow notions satisfy some trivial but very useful
transitivity properties. The idea is similar to the proof of the equivalence of
(3) and (4) in Fact \ref{fact1}. In the statement of the next
proposition we shall assume that all the sets under consideration are nonempty.

\begin{proposition} \labbel{trans}
If $\mu \Rightarrow ( \lambda _ j ) _{ j \in J } $ 
and
$ \lambda_j  \Rightarrow ( \upsilon  _{j, h}  ) _{ h \in H_j } $
for every $j \in J$, then
 $\mu  \Rightarrow ( \upsilon  _{j, h}  ) _{j \in J, h \in H_j } $
  \begin{enumerate}[(1)]
\item[(2)]
If 
$ \myarrow {\mu}{\kappa}{\lambda} $ and
$ \myarrow { \lambda }{ \theta }{ \Upsilon } $ 
 then
$ \myarrow {\mu}{\kappa \cdot \theta }{ \Upsilon} $
\item[(3)] 
If 
$ \myarrow {\mu}{\kappa}{\Lambda} $,
$\kappa \geq  \omega, | \Lambda |$ and
$ \myarrow { \lambda }{\kappa}{\Upsilon _ \lambda } $ 
for every $\lambda \in \Lambda $,
then
$ \myarrow {\mu}{\kappa}{\bigcup _{ \lambda \in \Lambda }  \Upsilon _ \lambda} $ 
\item[(4)]
If $\kappa$ is infinite, 
$ \myarrow {\mu}{\kappa}{\Lambda} $, 
$ \upsilon \in \Lambda$ and
$ \myarrow { \upsilon }{\kappa}{ \Lambda \setminus \{ \upsilon \}} $ 
then
$ \myarrow {\mu}{\kappa}{\Lambda \setminus \{ \upsilon \}} $. 
\item[(5)]
More generally, if $\kappa \geq  \omega, | \Upsilon | $, 
$ \myarrow {\mu}{\kappa}{\Lambda} $, 
$ \Upsilon \subset \Lambda$ and
$ \myarrow { \upsilon }{\kappa}{\Lambda \setminus \Upsilon } $ 
for every $ \upsilon \in \Upsilon$, 
then
$ \myarrow {\mu}{\kappa}{\Lambda \setminus \Upsilon} $. 
  \end{enumerate}
 \end{proposition}

\begin{proof}
(1) Just consider the compositions of the functions
given by
$\mu \Rightarrow ( \lambda _ j ) _{ j \in J } $ 
and
$ \lambda_j  \Rightarrow ( \upsilon  _{j, h}  ) _{ h \in H_j } $.

(2) and (3) are immediate from (1).

(4) Apply (1) by using the trivial relation
$ \myarrow { \lambda }{\kappa}{\lambda} $ 
for every $\lambda \not= \upsilon $.
That is, write
$ \myarrow {\mu}{\kappa}{\Lambda} $ 
as 
$\mu \Rightarrow ( \lambda _ j ) _{ j \in J } $
and,
for  $\lambda_j= \lambda \not= \upsilon$,
take $H _{j} $ a singleton $  \{  h_j \}  $ 
and $ \upsilon  _{j, h_j} = \lambda _j$. 

(5) Same as (4), 
by using 
$ \myarrow { \lambda }{\kappa}{\lambda} $ 
for every $\lambda \not\in \Upsilon $.
 \end{proof} 

\begin{theorem} \labbel{propert} (1) 
$ \myarrow {\mu}{1}{\cf \mu } $. 
\begin{enumerate}[(1)] \setcounter{enumi}{1}
 \item 
If $\mu $ is regular then
$ \myarrow {\mu^+}{\mu^+}{\mu} $. 
 \item
 More generally, if $\mu $ is regular then
$ \myarrow {\mu^{+n}}{\mu^{+n}}{\mu} $. 
\item 
If $\lambda$ is regular and $\cf\mu = \lambda ^{+n} $  then
$ \myarrow {\mu}{\cf\mu}{ \lambda } $. 
\item
If $\mu $ is singular then 
$ \myarrow {\mu^+}{\mu^+}{ \{ \cf \mu\} \cup \Lambda  } $,
for every $\Lambda \subseteq Reg_{< \mu}  $ cofinal in $\mu $. 
 \end{enumerate}
\end{theorem} 

\begin{proof}
(1) If $\mu $ is regular, this is trivial.
If $\mu $ is singular, 
let $(\mu _ \beta ) _{ \beta \in \cf \mu}  $
be an increasing  cofinal sequence in $\mu $, and consider
$f:  \mu  \to \cf \mu $ 
defined by $f( \alpha )=  
\inf \{ \beta \in \cf \mu \mid  \alpha \not \in \mu_ \beta  \} $.  

See \cite{bumi,arch}  for  proofs of (2) and (5), though with slightly different notations.

(3) follows from Proposition \ref{trans}(2)
and iterated applications of (2). Similarly,
(4) follows from (1), (3) and Proposition \ref{trans}(2). 
\end{proof}

More properties of 
$ \myarrow {\mu}{\kappa}{\lambda} $
can be obtained from Theorem \ref{propert} and  Proposition \ref{trans},
for example combining  \ref{propert}(3) and (5).

\begin{corollary} \labbel{correduc}
Suppose that $ \myarrow {\mu}{ \kappa }{ \Lambda } $
and let
$\Lambda' = 
\{ \lambda \in \Lambda \mid  \myarrownot { \lambda }{ \kappa }{ \Lambda \cap Card _{< \lambda } } \} $. 
(1) If $\kappa \geq \omega, |  \Lambda \setminus \Lambda '|$ then $ \myarrow {\mu}{ \kappa }{ \Lambda' } $.
Moreover, (2) if $ \lambda \in \Lambda'$ then 
 (a) if $\lambda$ is singular then $\cf\lambda \not\in \Lambda $;
(b)  if $\lambda= \xi^+$ and $ \xi$ is regular, then $\xi \not \in \Lambda   $. 
(c)  if $\lambda= \xi^+$ and $ \xi$ is singular then  either 
$\cf\xi \not \in \Lambda   $ or $Card(\xi', \xi) \cap \Lambda   = \emptyset $,
for some $ \xi' < \xi$.  
 \end{corollary} 

\begin{proof}
(1)
By Fact \ref{fact0}(2) without loss of generality $\mu \geq \sup \Lambda $.  

We first prove (1) under the additional assumption that
$\mu \not\in \Lambda $.
Fix any $\kappa \geq \omega $, suppose by contradiction that (1) fails
with respect to that $\kappa$,
 and consider a counterexample in which $\mu $ is of minimal cardinality.
Thus $ \myarrow {\mu}{ \kappa }{ \Lambda } $,
 $ \myarrownot {\mu}{ \kappa }{ \Lambda' } $ and
$\kappa \geq |  \Lambda \setminus \Lambda '|$.
Let $ \Upsilon = \Lambda \setminus \Lambda' = \{ \lambda \in \Lambda \mid  \myarrow { \lambda }{ \kappa }{ \Lambda \cap Card _{< \lambda } } \}$,
thus  $ \Lambda' = \Lambda \setminus \Upsilon$.

Suppose that $ \upsilon  \in \Upsilon $, hence $ \upsilon < \mu $, since 
$\mu \not \in \Lambda $.   By definition, 
$\myarrow { \upsilon  }{ \kappa }{ \Lambda \cap Card _{< \upsilon  } }$. 
Let $\Lambda_1= \Lambda \cap Card _{< \upsilon  }$
and consider the statement of the corollary  
 with $ \upsilon $ in place of $\mu $ and
$\Lambda_1$ in place of $\Lambda$.
Clearly, again by  Fact \ref{fact0}(2),
$\Lambda_1'  = \Lambda' \cap Card _{< \upsilon }$,
thus 
$   |  \Lambda_1 \setminus \Lambda_1 '|  \leq  |  \Lambda \setminus \Lambda '| \leq \kappa $. 
By the minimality of $\mu $, and since
$ \upsilon \not \in \Lambda_1$,  
we can apply (1) thus getting 
$\myarrow { \upsilon  }{ \kappa }{ \Lambda_1' }$.
Since
 $ \Lambda_1' \subseteq \Lambda' = \Lambda \setminus \Upsilon$,
we get
$\myarrow { \upsilon  }{ \kappa }{ \Lambda \setminus \Upsilon }$. 

In the previous paragraph we have proved that
if  $ \upsilon  \in \Upsilon $ then 
$\myarrow { \upsilon  }{ \kappa }{ \Lambda \setminus \Upsilon }$. 
Since 
$ \myarrow {\mu}{ \kappa }{ \Lambda } $, by Proposition \ref{trans}(5) 
we get 
$\myarrow { \upsilon  }{ \kappa }{ \Lambda \setminus \Upsilon }$,
a contradiction, since $ \Lambda \setminus \Upsilon = \Lambda '$. 
We have proved (1) under the assumption that
$\mu \not\in \Lambda $.

Now suppose that $\mu \in \Lambda $. 
If $\mu \in \Lambda' $ then trivially   $ \myarrow {\mu}{ \kappa }{ \Lambda' } $.
Otherwise $\mu \not\in \Lambda' $, then, by definition,
 $\myarrow{ \mu  }{ \kappa }{ \Lambda \cap Card _{< \mu  } }$.
Letting $ \Lambda_2 =  \Lambda \cap Card _{< \mu  } $,
and since $\mu \not \in \Lambda_2$, we can apply the already
proved particular case, getting 
 $ \myarrow {\mu}{ \kappa }{ \Lambda_2 ' } $.  
But $ \Lambda_2 '  \subseteq  \Lambda' $,
thus  $ \myarrow {\mu}{ \kappa }{ \Lambda' } $.  
The proof of (1) is complete.

(2) is immediate from Theorem \ref{propert}.  
 \end{proof}  

\section{Filters, partitions} \labbel{partbo}

The principle $\mu \Rightarrow ( \lambda _ j ) _{ j \in J } $ 
admits a characterization in terms of filters 
(not necessarily ultra).
By a \emph{filter} we shall always mean a \emph{proper}
filter, that is $ \emptyset \not\in F$.  
If $F$ is a filter over some set $I$
and $f: I \to H$ is a function,
we denote by
$f(F)$ the filter over $H$ defined by
$Y \in f(F)$ if and only if 
$f ^{-1}(Y) \in F $. 
This extends the notation introduced for ultrafilters and 
 is connected with the \emph{Katetov order},
as we shall briefly discuss in Section \ref{kat}.
For every cardinal $\mu $, we shall denote by
$F_\mu$ the filter consisting of all subsets
$A$ of $\mu $ such that $| \mu \setminus A| < \mu$.
A filter $F$   over $\mu $ is \emph{uniform}
if all members of $F$ have cardinality $\mu $. We say
that  $F$ is \emph{strongly uniform} if
$F \supseteq F_ \mu$. Equivalently,
$F$ is strongly uniform if and only if    
 every filter extending $F$ is uniform
(since we are taking into account only proper filters).
Notice that an ultrafilter over $\mu $ is uniform
if and only if it is strongly uniform.
All the above definitions and properties
hold also when dealing with a field of subsets 
of $\mu $ $\mathcal F$  containing $F_ \mu$. 

\begin{proposition} \labbel{filt}
For every  sequence  
of functions 
 $f_ j  : \mu  \to \lambda _ j  $
($j \in J$),
the following conditions are equivalent. 
\begin{enumerate}[(1)]
   \item
The $f_ j $'s
witness $\mu \Rightarrow ( \lambda _ j ) _{ j \in J } $,
that is, 
for every  uniform ultrafilter $D$ over $ \mu $ 
there is $j \in J$ such that 
$f_ j (D)$ is uniform over $ \lambda _ j$.    
\item
For every strongly  uniform filter $F$ over $ \mu $ 
there is $j \in J$ such that 
$f_ j (F)$ is uniform over $ \lambda _ j$.    
\item
For every  family
 $(B_j) _{j \in J} $
such that  
$B_j  \subseteq  \lambda_j$ and 
$| B_j | < \lambda_j$,
for $j \in J$,
there is a finite set $N \subseteq J$
such that  
$| \bigcap_{j \in N} f_j ^{-1} (B_j) | < \mu$.
\end{enumerate}
 \end{proposition}

\begin{proof} 
(2) $\Rightarrow $  (1) is trivial.

(1) $\Rightarrow $  (2) Suppose that (1)
holds, and that $F$ is a strongly uniform filter over $\mu $.
Extend $F$ to an ultrafilter $D$; thus $D$ is  uniform over $\mu $, hence 
by (1)
 $f_ j (D)$ is uniform over $ \lambda _ j$, for some $j \in J$.
Since  $f_ j (F) \subseteq f_ j (D)$, then also
  $f_ j (F)$ is uniform over $ \lambda _ j$.
Of course here we are heavily using the Axiom of Choice, 
at least in its weaker incarnation as the Prime Ideal Theorem. 

(2) $\Leftrightarrow $  (3) We shall prove the
equivalence of the negations.
The negation of (2) means that there is 
a strongly uniform filter $F$ over $\mu $ 
such that for every $j \in J$ the filter $f_j(F)$
is not uniform over $\lambda_j$,
that is, there is $B_j  \in  f_j(F)$
such that $| B_j | < \lambda_j$.
A filter $F$ as in the previous sentence exists 
if and only if 
there is  a family
 $(B_j) _{j \in J} $
such that  
$B_j  \subseteq  \lambda_j$, 
$| B_j | < \lambda_j$,
for $j \in J$, and 
$\{  f_j ^{-1} (B_j) \mid j \in J \} \cup F_ \mu$
has the finite intersection property. 
This is exactly the negation of (3).
\end{proof} 

We can also equivalently state 
$\mu \Rightarrow ( \lambda _ j ) _{ j \in J } $ 
in an ``internal way  on $\mu $''
by using partitions. In view of Fact \ref{fact1}, 
this generalizes
Remark \ref{rmkmeas}(1). 

\begin{lemma} \labbel{partitions1} 
Suppose that $\mu \geq \lambda, \lambda _j$, for $j \in J$.
  \begin{enumerate}[(1)]
 \item   
$\mu \Rightarrow ( \lambda _ j ) _{ j \in J } $
if and only if:
 there is a sequence $(\pi_ j ) _{ j \in J }$
of partitions of $\mu $ 
such that each $\pi_j$ has $\lambda_j$ classes 
and, for every uniform ultrafilter $D$ over $\mu $, 
there is $j \in J$ such that no union
of $<\lambda_j$ classes of $\pi_j$ belongs to $D$. 
\item
In particular, 
$ \myarrow {\mu}{\kappa}{\lambda} $ 
if and only if 
there is a sequence $(\pi_ \gamma  ) _{ \gamma \in \kappa  }$
of partitions of $\mu $ into
$\lambda$ classes 
such that,  for every uniform ultrafilter $D$ over $\mu $, 
there is $ \gamma \in \kappa $ such that no union
of $<\lambda$ classes of $\pi _ \gamma $ belongs to $D$. 
  \end{enumerate}
\end{lemma} 

 \begin{proof}
(1) Let
 $\mu \Rightarrow ( \lambda _ j ) _{ j \in J } $
be witnessed by $(f_ j ) _{ j \in J }$.
By Fact \ref{fact0},
then $(f_ j ) _{ j \in J' }$ witnesses
 $\mu \Rightarrow ( \lambda _ j ) _{ j \in J' } $,
where
$J' = \{ j \in J \mid |f_j(\mu)|= \lambda _j \}$. 
To each $f_j$ $(j \in J')$ there is naturally 
associated a partition $\pi_j$ of $\mu $ into $\lambda_j$ 
classes in such a way that 
the sufficient condition is satisfied with $J'$ in place of $J$.
Letting $\pi_j$ be arbitrary  with $|\lambda_j|$ many classes
for   $j \in J \setminus J'$ we get the condition.

Conversely, given $(\pi_ j ) _{ j \in J }$ partitions as in the sufficient condition,
 enumerate the classes of each $\pi_j$ and consider
the corresponding functions
 $f_ j  : \mu  \to \lambda _ j  $.
These witness $\mu \Rightarrow ( \lambda _ j ) _{ j \in J } $.

(2) follows trivially from the definition of $ \myarrow {\mu}{\kappa}{\lambda} $.
 \end{proof} 

If $(\pi_ j ) _{ j \in J }$ are partitions as in 
Lemma \ref{partitions1}(1) 
 we shall say that the $\pi_ j $'s
\emph{witness} $\mu \Rightarrow ( \lambda _ j ) _{ j \in J } $.
Though essentially trivial, 
 Lemma \ref{partitions1}
is useful, since  certain properties (e.~g., transitivity) of 
 $\mu \Rightarrow ( \lambda _ j ) _{ j \in J } $
 are best seen in terms of Definition \ref{def},
while  Lemma \ref{partitions1}
is useful if we want to work ``inside $\mu $'',
as we shall do in the next propositions.
Of course, this remark is nothing but a variant on the old story
of viewing homomorphic images as 
corresponding to equivalence relations.

For $(\pi_ j ) _{ j \in J }$  a sequence of partitions of $\mu $ 
such that each $\pi_j$ has $\lambda_j$ classes,
let $\mathcal F _ \mu (\pi_ j ) _{ j \in J }$  be the smallest field
of subsets of $\mu $ which
contains $F_ \mu$ and which,  for every $j \in J$,
contains all unions
of $<\lambda_j$ classes of $\pi_j$.

\begin{proposition} \labbel{partitions2} 
Suppose that $(\pi_ j ) _{ j \in J }$ 
is a sequence 
of partitions of $\mu $ 
such that each $\pi_j$ has $\lambda_j$ classes. 
Then the  following conditions are equivalent.  
\begin{enumerate}[(1)]   
\item
The $\pi_j$'s 
witness $\mu \Rightarrow ( \lambda _ j ) _{ j \in J } $
that is for every uniform ultrafilter $D$ over $\mu $ 
there is $j \in J$ such that no union
of $<\lambda_j$ classes of $\pi_j$ belongs to $D$.
\item
For every strongly uniform filter $F $ over $\mu $  
there is $j \in J$ such that no union
of $<\lambda_j$ classes of $\pi_j$ belongs to $F$.
\item
For every choice of subsets $A_j$ of $\mu $,
one for each $j \in J$,  
such that each $A_j$ is a union
of  $<\lambda_j$ classes of $\pi_j$,
 there is a finite set $N \subseteq J$ 
such that  $|\bigcap _{j \in n} A_j | < \mu $.
\item
For every strongly uniform filter $F$ 
over $\mathcal F _ \mu (\pi_ j ) _{ j \in J }$ 
there is $j \in J$ such that no union
of $<\lambda_j$ classes of $\pi_j$ belongs to $F$.
\item
For every  uniform ultrafilter $D$ 
over $\mathcal F _ \mu (\pi_ j ) _{ j \in J }$
there is $j \in J$ such that no union
of $<\lambda_j$ classes of $\pi_j$ belongs to $D$.
  \end{enumerate} 
 \end{proposition}

 \begin{proof}
The equivalences of (1)-(3) and of (4)-(5) are entirely similar to the proof of
Proposition \ref{filt} (the assumption that 
 $\mathcal F _ \mu (\pi_ j ) _{ j \in J } \supseteq F_ \mu$
is used in (4) $\Leftrightarrow $  (5)).

Then notice that (2) and (4) are equivalent since both conditions are actually evaluated in
$\mathcal F _ \mu (\pi_ j ) _{ j \in J }$, that is,
a filter $F$ over $\mu $ satisfies (2) if and only if 
$F \cap \mathcal F _ \mu (\pi_ j ) _{ j \in J }$ satisfies (4).
\end{proof}

\section{Topological equivalents} \labbel{top} 

If $X$ is a topological space 
and $Y $ is an infinite subset of $X$, 
a point $x \in X$ is a \emph{complete accumulation point}
of $Y$ if $| Y \cap U| = |Y|$, for every neighborhood 
$U$ of $x$.
The space $X$ is \emph{$\mu $-compact}
if every subset of cardinality $\mu $ has a complete accumulation point.
In the literature $\mu $-compactness has also been given various
other disparate names, such as $\CAP_ \mu $,
$C( \mu,\mu)$,  \emph{$[ \mu, \mu ]$-compactness in the sense of accumulation points},
etc. 

It is convenient to introduce also a slight modification dealing
with sequences rather than subsets.
If $(x_ \alpha ) _{ \alpha \in \mu } $ 
is a sequence (possibly with repetitions) of elements of $X$,
a point $x \in X$ is a \emph{$\mu $-complete accumulation point}
of $Y$ if $| \{ \alpha \in \mu \mid  x_ \alpha \in U\}|= \mu  $, for every neighborhood 
$U$ of $x$.
The space $X$ is \emph{$\mu $-${^*}$compact}
if every sequence $(x_ \alpha ) _{ \alpha \in \mu } $
of elements of $X$ 
 has a $\lambda$-complete accumulation point.
 
The above notions are connected by the following easy proposition,
which shows that the compactness 
notions are distinct only when $\mu $ is singular.
See \cite[VI, Proposition 1]{nuotop} 
or \cite[Proposition 3.3]{tproc2} 
for details. The latter proposition
is stated in a more general framework, the present case
is when $\mathcal F$ is the set of all singletons of $X$.

\begin{proposition} \labbel{singreg}
(1) If $\mu $ is regular then a topological space is 
$\mu $-$^*$compact if and only if it is 
$\mu $-compact.

(2) A topological space is 
$\mu $-$^*$compact if and only if it is 
both $\mu $-compact
and $\cf\mu $-compact.
\end{proposition}

\begin{theorem} \labbel{cpntop} 
Suppose that  each $\lambda_j$ is a regular cardinal, 
 endowed either with the order topology or with the
left order topology.
The space
$ \prod _{j \in J} \lambda_j $ is $\mu $-$^*$compact if and only if
$\mu \not\Rightarrow ( \lambda _ j ) _{ j \in J } $
\end{theorem}

We are going to show that,
under some weak and natural hypotheses,
and as far as products of cardinals are concerned,
$\mu $-$^*$compactness and $\mu $-compactness
are equivalent. Of course, for $\mu $ regular, the next corollary reduces to 
Theorem \ref{cpntop}.

\begin{corollary} \labbel{cpntopcor} 
Suppose that  $(\lambda_j) _{j \in J} $ is a sequence of regular cardinals,
and that each cardinal occurring in the sequence occurs infinitely many times, that is,
$|\{ h \in J \mid \lambda_h = \lambda _j  \}| \geq \omega $, for every $j \in J$. 
 If each $\lambda_j$ is
 endowed either with the order topology or with the
left order topology, then
$ \prod _{j \in J} \lambda_j $ is $\mu $-compact if and only if
either (a)
$ |\prod _{j \in J} \lambda_j| < \mu $, or (b)
$ |\prod _{j \in J} \lambda_j| \geq \mu $ and
$\mu \not\Rightarrow ( \lambda _ j ) _{ j \in J } $.

In particular, if 
$ |\prod _{j \in J} \lambda_j| \geq \mu $,
then
$ \prod _{j \in J} \lambda_j$
is $\mu $-compact if and only if it is
$\mu $-$^*$compact.
\end{corollary}  

\begin{proof}
If (a) holds then 
$ \prod _{j \in J} \lambda_j$
is vacuously  $\mu $-compact.

If (b) holds then
$ \prod _{j \in J} \lambda_j$
is $\mu $-$^*$compact
by Theorem \ref{cpntop}, hence 
  $\mu $-compact by Proposition \ref{singreg}.

Conversely,  suppose that
$ \prod _{j \in J} \lambda_j$ is
  $\mu $-compact
and 
$ |\prod _{j \in J} \lambda_j| \geq \mu $.
Since each $\lambda_j$ occurs infinitely many times in the sequence, 
then
$ \prod _{j \in J} \lambda_j$ is
homeomorphic 
to 
$Y = \prod _{j \in J} \lambda_j \times \prod _{j \in J} \lambda_j$,
in particular, $Y$ is   $\mu $-compact.
Let $(x_ \alpha ) _{ \alpha \in \mu} $
be a sequence of elements  of
$ \prod _{j \in J} \lambda_j$, and let
$(y_ \alpha ) _{ \alpha \in \mu} $
be a sequence of \emph{distinct}
elements of 
$ \prod _{j \in J} \lambda_j$.
Such a sequence exists since 
$| \prod _{j \in J} \lambda_j| \geq \mu$.  
 Then all elements
of the sequence
$(x_ \alpha , y_ \alpha )_{ \alpha \in \mu} $
in $Y$ are distinct.
By  $\mu $-compactness of $Y$,
the sequence has an accumulation point, 
say $(x,y)$.
Then $x$ is a $\lambda$-accumulation point of
$(x_ \alpha ) _{ \alpha \in \mu} $ in
 (the first copy of) $ \prod _{j \in J} \lambda_j$. In conclusion,
$ \prod _{j \in J} \lambda_j$ is 
$\mu $-$^*$compact, hence
$\mu \not\Rightarrow ( \lambda _ j ) _{ j \in J } $
by Theorem \ref{cpntop}.
\end{proof}

\begin{corollary} \labbel{topcpnL}
Suppose that $ \Lambda$ is a set of regular cardinals and $\kappa$ is infinite.
Then 
 $ \myarrownot {\mu}{\kappa}{\Lambda} $ 
if and only if 
$ \prod _{ \lambda \in \Lambda}  \lambda ^ \kappa $ is 
 $\mu $-$^*$compact.
If, in addition,
$ |\prod _{ \lambda \in \Lambda}  \lambda ^ \kappa | \geq \mu$ 
then the above conditions hold if and only if 
$ \prod _{ \lambda \in \Lambda}  \lambda ^ \kappa $ is 
 $\mu $-compact.
 \end{corollary}

\section{Model-theoretical equivalents} \labbel{mt} 

The main idea from \cite{bumi} 
for applying ``classical''  model theory to extended  logics 
was the introduction of the notion
of a $\mu $-nonstandard element
(said to \emph{bound} $\mu $ in \cite{bumi})
in models with a fixed order of type $\mu $---without
 loss of generality we can take this order to 
be $\mu $ itself.
The principle corresponding to
$ \myarrow {\mu}{\kappa}{\lambda} $,
for $\mu $, $\lambda$ regular, 
was then introduced there (with the same notation
when $\kappa> \mu$ and with $\kappa$ omitted when $\kappa= \mu $)
by asserting that every model with
a $\mu $-nonstandard element has some
 $\lambda$-nonstandard element,
modulo a theory of cardinality $\kappa$.

In details, let
$\m A = \langle \mu, <, \alpha , \dots  \rangle _{ \alpha \in \mu } $.
If $\m B$ is \emph{elementarily equivalent}
to $\m A$, in symbols 
 $\m B \equiv \m A$, an element  $b $ of $  B$ is
 \emph{$\mu$-nonstandard} 
if $ \alpha < b$ holds in $\m B$, for every $\alpha \in \mu $.
Of course, in  case $\mu= \omega $
we get the usual notion of a nonstandard element.
Compare also  \cite[pp. 116--118]{cha67}.
Similarly, for $\lambda < \mu $, an element 
$c $ of $ B$ is \emph{$\lambda$-nonstandard} 
if $  c < \lambda $ and $\beta < c $ hold  in $\m B$, for every $\beta \in \lambda$.

\begin{thm} \labbel{nonstL}
If $\mu $ is regular,
$\Lambda$ is a set of regular cardinals
and $ \kappa \geq \mu \geq \sup \Lambda  $,  then 
$ \myarrow {\mu}{\kappa}{\Lambda} $ 
if and only if there is an
 expansion $\m A$ of 
$\langle \mu, <, \alpha \rangle _{ \alpha \in \mu } $
with at most $\kappa$ new symbols
(equivalently, symbols and sorts) 
such that for every 
$\m B \equiv \m A$, 
if $\m B$ has a $\mu$-nonstandard element
then $\m B$ has a $\lambda$-nonstandard element,
for some $ \lambda \in \Lambda$.
 \end{thm}

\begin{proof}
The proof is an immediate generaliztion of 
\cite[Theorem 4]{cpnw}.
Full details appeared in 
\cite{nuotop}. 
\end{proof}

There are possible generalizations of Theorem \ref{nonstL} 
to singular cardinals, but they are involved, quite technical
and, so far, appear to be of little practical use.
We address the interested reader to \cite{nuotop}.

\section{The trace of Kat{\v{e}}tov order} \labbel{kat}

\begin{definition} \labbel{defkat}    
Suppose that $F$ is a filter over $I$ 
and, for each $j \in J$, $G_j$ is a filter over $H_j$. 
We write
$F \Rightarrow ( G_ j ) _{ j \in J } $ 
in case the following statement holds.
 \begin{enumerate}[(1)]   
\item[]
There is a sequence  
$(f_ j ) _{ j \in J }$
of functions 
such that  $f_ j  :I   \to H _ j  $
for $j \in J$, and such that 
for every  ultrafilter $D$ over $I $
containing $F$  
there is $j \in J$ such that 
$f_ j (D)$ contains $ G _ j$.    
 \end{enumerate}

More generally, 
if $\mathcal K \subseteq \mathcal P(J)$,
we let 
$F \Rightarrow_{\mathcal K} ( G_ j ) _{ j \in J }$ 
mean the following.
 \begin{enumerate}[(1)]   
\item[]
There is a sequence  
$(f_ j ) _{ j \in J }$
of functions 
such that  $f_ j  :I   \to H _ j  $
for $j \in J$, and such that
$\{ j \in J \mid  f_ j (D) \supseteq  G _ j \} \in \mathcal K $ 
for every  ultrafilter $D$ over $I $
containing $F$.    
 \end{enumerate}

Thus 
$F \Rightarrow ( G_ j ) _{ j \in J } $ 
is
$F \Rightarrow_{\mathcal K} ( G_ j ) _{ j \in J } $ 
when 
$\mathcal K $ is
the set of all one-element subsets of $J$. 
As in Definition \ref{def}, 
$ \myarrow {F}{\kappa}{G} $ 
denotes the case when 
$|J |= \kappa $ and the $ G _ j $'s  are all equal to $G $.
The 
negations of the principles are denoted by
$F \not\Rightarrow ( G_ j ) _{ j \in J } $
and
$ \myarrownot {F}{\kappa}{G} $,
everything possibly with
the subscript ${\mathcal K}$  added.

We can also consider variants of the above definitions 
in which only $\kappa$-complete ultrafilters $D$ are taken into account;
we shall denote such modified notions  as
$F \Rightarrow ( G_ j ) _{ j \in J }\ (\kappa\text{-complete})$,
$ F \Rightarrow^ \kappa _{\mathcal K}  G \ (\kappa\text{-complete})$
 and  so on.
\end{definition}

\bibliographystyle {asl}

\def\cprime{$'$} \def\cprime{$'$} \def\cprime{$'$}

\end{document}